\documentclass[11pt]{article}
\usepackage[T1]{fontenc}
\usepackage{lmodern}
\usepackage[a4paper,margin=27mm]{geometry}
\usepackage{amsmath,amssymb,amsthm}
\usepackage{mathtools}
\usepackage{microtype}
\usepackage{booktabs,longtable,array}
\usepackage{xcolor}
\usepackage{xurl}
\usepackage[colorlinks=true,linkcolor=blue!45!black,citecolor=blue!45!black,urlcolor=blue!45!black]{hyperref}
\hypersetup{pdftitle={Certified Fourier rounding for the real Grothendieck inequality},pdfauthor={Bo Peng}}
\numberwithin{equation}{section}
\newtheorem{theorem}{Theorem}[section]
\newtheorem{proposition}[theorem]{Proposition}

\theoremstyle{definition}

\newcommand{\R}{\mathbb R}
\newcommand{\N}{\mathbb N}
\newcommand{\E}{\mathbb E}
\newcommand{\He}{\operatorname{He}}
\newcommand{\erf}{\operatorname{erf}}
\newcommand{\erfc}{\operatorname{erfc}}
\newcommand{\sgn}{\operatorname{sgn}}
\newcommand{\KG}{K_G^{\R}}

\DeclareMathOperator{\Odd}{Odd}
\title{Certified Fourier rounding for the\\real Grothendieck inequality}
\author{Bo Peng\thanks{\raggedright Institute of Mathematical Sciences, ShanghaiTech University.
Correspondence: \href{mailto:pengbo2025@shanghaitech.edu.cn}{pengbo2025@shanghaitech.edu.cn}.}}
\date{October 17, 2026}

\begin{document}
\maketitle
\begin{abstract}
We give a computer-assisted proof that the real Grothendieck constant is
strictly smaller than \(1.77987164\). The proof provides a complete
interval-arithmetic realization of Ren's five-variable construction.
A finite Fourier calculation encloses every scalar coefficient through
degree \(403\); a fourth-derivative estimate on the entire unit circle
bounds the remaining coefficient sum by \(10^{-4}\).
Exact rational arithmetic then certifies the final inequality.
We include the transfer to all finite real matrices and arbitrary vector
dimensions, the analytic quadrature errors, the rational parameters, and
source code for checking and regenerating the certificate.
\end{abstract}

\section{Introduction}

Grothendieck's inequality compares the value of a real bilinear form on
unit vectors with its value on signs~\cite{Grothendieck1953}.
For a finite real matrix \(M=(M_{ij})\in\R^{m\times n}\), put
\begin{align}
 S(M)&=\max_{\epsilon_i,\delta_j\in\{-1,1\}}
       \left|\sum_{i=1}^m\sum_{j=1}^nM_{ij}\epsilon_i\delta_j\right|,
       \label{eq:sign-value}\\
 V(M)&=\sup_{d\ge1}\ \max_{x_i,y_j\in\mathbb S^{d-1}}
       \left|\sum_{i=1}^m\sum_{j=1}^nM_{ij}\langle x_i,y_j\rangle\right|.
       \label{eq:vector-value}
\end{align}
The real Grothendieck constant \(\KG\) is the least number \(K\) such
that \(V(M)\le K S(M)\) for every finite real matrix \(M\).

Krivine proved
\[
 \KG\le\frac{\pi}{2\log(1+\sqrt2)}
       =1.782213978\ldots
\]
by preprocessing vector correlations before Gaussian sign
rounding~\cite{Krivine1979}. Braverman, Makarychev, Makarychev and
Naor proved a strict improvement over this bound~\cite{Braverman2013}.
Naor and Regev subsequently established the optimality of Krivine
schemes as a class~\cite{NaorRegev2014}.
Song and Yue gave the explicit bound
\begin{equation}\label{eq:song-yue}
 \KG\le\frac{625000\pi}{1102769}
       =1.780513787\ldots
\end{equation}
in their preprint~\cite{SongYue2026}. Their route combines a nonlinear
Gaussian rounding kernel, a coefficient majorant, and a certified
infinite tail.

Ren's public note and implementation~\cite{Ren2026} introduced the
five-variable sign function studied here and a degree-\(81\)
preprocessing polynomial. The note reports the stronger value
\(1.779754412112\) using a localized fourth- and sixth-order tail
estimate. Our contribution is a complete
interval-arithmetic realization of that construction. We evaluate the
full coefficient head, supply a whole-circle fourth-order tail estimate,
and combine them by an exact rational comparison. The precise result is
as follows.

\begin{theorem}\label{thm:main}
The real Grothendieck constant satisfies
\[
 \boxed{\displaystyle
 \KG<\frac{44496791}{25000000}=1.77987164<1.7799.}
\]
Consequently, for all positive integers \(m,n,d\), every
\(M\in\R^{m\times n}\), and all unit vectors
\(x_1,\ldots,x_m,y_1,\ldots,y_n\in\R^d\),
\[
 \left|\sum_{i,j}M_{ij}\langle x_i,y_j\rangle\right|
 \le\frac{44496791}{25000000}
       \max_{\epsilon_i,\delta_j\in\{-1,1\}}
       \left|\sum_{i,j}M_{ij}\epsilon_i\delta_j\right|.
\]
\end{theorem}

The numerical proof has two outputs. If
\(H(z)=\sum_{n\ge1}c_nz^n\) is the scalar kernel defined below, then
\begin{align}
 c_1-\sum_{\substack{3\le n\le403\\n\text{ odd}}}|c_n|
     &>\frac{88263349}{10^8}, \label{eq:head-result}\\
 \sum_{\substack{n>403\\n\text{ odd}}}|c_n|
     &<\frac1{10000}. \label{eq:tail-result}
\end{align}
A general rounding argument converts their difference into an upper
bound on \(\KG\). Sections~\ref{sec:construction} and
\ref{sec:transfer} give this argument. Sections~\ref{sec:head} and
\ref{sec:tail} establish the two estimates, and
Section~\ref{sec:conclusion} completes the exact comparison.

\section{The five-variable kernel}\label{sec:construction}

Let \(\psi_j=\He_j/\sqrt{j!}\) be the orthonormal probabilists'
Hermite basis for standard Gaussian measure, and write
\(\psi_\beta(x)=\prod_{i=1}^4\psi_{\beta_i}(x_i)\) for
\(\beta\in\N^4\). Here \(\N\) includes zero.
The rational parameters in Appendix~\ref{app:parameters} define
\begin{align}
 B(x)&=\ell x+\sum_{k=1}^6b_k\sin(3kx/4),&
 S(x)&=\sum_{i=1}^4B(x_i),\label{eq:boundary-local}\\
 U(s)&=s+\sum_{k=1}^4a_k\sin(2ks),&
 f(w,x)&=\sgn\bigl(w+U(S(x))\bigr).\label{eq:boundary-outer}
\end{align}
The value assigned to the sign at zero is immaterial: conditioning on
\(x\) shows that the boundary has Gaussian measure zero.
All positive-order derivatives of \(U\circ S\) are bounded, and
\(U\circ S\) has at most linear growth.

For independent standard normal variables \(W,X_1,\ldots,X_4\), set
\begin{equation}\label{eq:hermite-coeff}
 A_{j,\beta}=\E\bigl[f(W,X)\psi_j(W)\psi_\beta(X)\bigr].
\end{equation}
Parseval and \(|f|=1\) give \(\sum_{j,\beta}A_{j,\beta}^2=1\).
Since \(f(-w,-x)=-f(w,x)\), the coefficient \(A_{j,\beta}\) vanishes
when \(j+|\beta|\) is even. Define
\begin{equation}\label{eq:bivariate-kernel}
 \mathcal H(a,b)=\frac\pi2\sum_{j,\beta}
                       A_{j,\beta}^2a^jb^{|\beta|}.
\end{equation}
The preprocessing polynomial and its exact coefficient norm are
\begin{equation}\label{eq:preprocessing}
 P(z)=\sum_{k=1,3,\ldots,81}p_kz^k,
 \qquad q_0:=\sum_k|p_k|=\frac{999999999999}{10^{12}}<1.
\end{equation}
We use the scalar kernel
\begin{equation}\label{eq:scalar-kernel}
 H(z)=\mathcal H(P(z),-z)=\sum_{n\ge1}c_nz^n.
\end{equation}
It has real coefficients and positive odd powers. If \(\|\cdot\|_1\)
denotes the absolute sum of power-series coefficients, then
\begin{equation}\label{eq:absolute-convergence}
 \sum_n|c_n|
 \le\frac\pi2\sum_{j,\beta}A_{j,\beta}^2q_0^j
 \le\frac\pi2.
\end{equation}
Thus the defining series is absolutely convergent on the closed unit
disk, including its two real endpoints.

\section{Transfer to all finite matrices}\label{sec:transfer}

The next proposition isolates the universal part of the proof.
Its five Gaussian coordinates are the channels used for rounding.
The input vectors may have any finite dimension.

\begin{proposition}[Coefficient majorant]\label{prop:transfer}
For the kernel in~\eqref{eq:scalar-kernel}, suppose that
\[
 \Gamma:=c_1-\sum_{n\ge3}|c_n|>0.
\]
Then \(\KG\le\pi/(2\Gamma)\).
\end{proposition}

\begin{proof}
Fix \(M\in\R^{m\times n}\). Every family of \(m+n\) vectors spans
a space of dimension at most \(m+n\). The positive semidefinite
\((m+n)\)-by-\((m+n)\) matrices with diagonal entries one form a
compact set, and every such matrix admits a real Gram realization.
Consequently the supremum in~\eqref{eq:vector-value} is attained.
Choose a maximizing family and, if necessary, negate all left vectors
so that, with \(t_{ij}=\langle x_i,y_j\rangle\),
\begin{equation}\label{eq:maximizing-vectors}
 \sum_{i,j}M_{ij}t_{ij}=V(M).
\end{equation}
For every positive integer \(k\), tensor powers give another unit-vector
family. Hence
\begin{equation}\label{eq:tensor-comparison}
 \left|\sum_{i,j}M_{ij}t_{ij}^{k}\right|\le V(M).
\end{equation}

Introduce two unit feature maps
\begin{align}
 L_P(x)&=\bigoplus_k\sqrt{|p_k|}\,x^{\otimes k}
                      \oplus\sqrt{1-q_0}\,e_L,\nonumber\\
 R_P(y)&=\bigoplus_k\sgn(p_k)\sqrt{|p_k|}\,y^{\otimes k}
                      \oplus\sqrt{1-q_0}\,e_R,\label{eq:features}
\end{align}
where \(e_L,e_R\) are mutually orthogonal and orthogonal to all tensor
summands. Their cross inner product is
\(\langle L_P(x),R_P(y)\rangle=P(\langle x,y\rangle)\).
A common standard Gaussian vector in this finite feature space gives
all the primary projections \(W_i,W'_j\).
Four further independent standard Gaussian vectors in the original
vector space give
\[
 X_{i,r}=\langle g_r,x_i\rangle,
 \qquad Y_{j,r}=-\langle g_r,y_j\rangle,
 \qquad 1\le r\le4.
\]
At each endpoint the five coordinates are independent standard normals.
For each pair \((i,j)\), their cross correlations are
\(P(t_{ij}),-t_{ij},-t_{ij},-t_{ij},-t_{ij}\).

Set \(\epsilon_i=f(W_i,X_i)\) and
\(\delta_j=f(W'_j,Y_j)\). These are jointly defined signs for the
whole matrix. The Hermite covariance identity yields
\begin{equation}\label{eq:rounding-covariance}
 \E[\epsilon_i\delta_j]
 =\sum_{j',\beta}A_{j',\beta}^2P(t_{ij})^{j'}(-t_{ij})^{|\beta|}
 =\frac2\pi H(t_{ij}).
\end{equation}
This identity also holds at degenerate correlations. Indeed, finite
Hermite expansions converge in \(L^2\) under both standard Gaussian
marginals, and Cauchy--Schwarz controls their product error.

Every realization of the signs has objective at most \(S(M)\).
Using~\eqref{eq:absolute-convergence} to interchange the absolutely
convergent series with the finite matrix sum, followed by
\eqref{eq:maximizing-vectors} and~\eqref{eq:tensor-comparison}, gives
\begin{align*}
 S(M)&\ge\E\sum_{i,j}M_{ij}\epsilon_i\delta_j\\
 &=\frac2\pi\sum_{n\ge1}c_n\sum_{i,j}M_{ij}t_{ij}^n\\
 &\ge\frac2\pi\left(c_1-\sum_{n\ge3}|c_n|\right)V(M).
\end{align*}
This proves the proposition.
\end{proof}

\section{Certification of the finite coefficient head}\label{sec:head}

We compute through degree \(N=403\). Since \(P(0)=0\), only
Hermite indices with \(j+|\beta|\le N\) can contribute to these scalar
degrees. The finite calculation uses Fourier separation of the four
identical auxiliary coordinates; polynomial products collect their
complete ordered multiindex sum.

\subsection{Finite Fourier formula}

Write \(\phi(u)=e^{-u^2/2}/\sqrt{2\pi}\), and let
\(q_j(s)=\E[\sgn(W+s)\psi_j(W)]\).
For \(j\ge1\),
\begin{equation}\label{eq:conditional-fourier}
 q_j(s)=\sqrt{2/\pi}\,i^{j-1}\int_{\R}V_j(u)e^{ius}\,du,
 \qquad V_j(u)=\frac{\phi(u)u^{j-1}}{\sqrt{j!}}.
\end{equation}
For \(j=0\), use the symmetric principal value with
\(V_0(u)=\phi(u)/u\), fixed by
\(q_0(s)=\erf(s/\sqrt2)\).
Put \(A(s)=U(s)-s\), and expand the real phase
\(e^{iuA(s)}\) through degree \(M=31\).
The exact Laurent coefficients
\begin{equation}\label{eq:phase-laurent}
 \frac1{m!}\left(\sum_{k=1}^4\frac{a_k}2
                         (\zeta^k-\zeta^{-k})\right)^m
 =\sum_kr_{m,k}\zeta^k
\end{equation}
give the modified Fourier profile
\begin{equation}\label{eq:modified-profile}
 \widetilde V_j(u)=\sum_{m=0}^{31}\sum_k
                    r_{m,k}(u-2k)^mV_j(u-2k).
\end{equation}
For \(j=0\), the sole uncancelled principal-value singularity is the
\(m=0\) term at zero.

For real \(u\), the characteristic coefficient
\[
 p_k(u):=i^{-k}\E[\psi_k(X)e^{iuB(X)}]
\]
is real, by the parity of \(B\). Let \(\widetilde p_k(u)\) be its
finite spatial quadrature with nodes \(v/256\),
\(-5120\le v\le5120\). The Gaussian weight and the normalization
\(1/\sqrt{k!}\) are included in this quadrature.
Use \(\Delta=2\pi/40\) and the positive frequencies
\(u_n=n\Delta\), \(1\le n\le F=768\). Define the polynomials
\begin{align}
 K_{uv}(z)&=\sum_{k=0}^{403}\widetilde p_k(u)
                                      \widetilde p_k(v)z^k,\nonumber\\
 W_{uv}(z)&=\sum_{j=0}^{403}\widetilde V_j(u)
                                      \widetilde V_j(v)z^j.\label{eq:pair-polynomials}
\end{align}
Writing \(\Odd_{\le N}\) for restriction to positive odd degrees at
most \(N\), the nonzero-frequency contribution is
\begin{equation}\label{eq:finite-head}
 4\Delta^2\Odd_{\le403}\sum_{u,v\in\{u_1,\ldots,u_F\}}
                W_{uv}(P(z))K_{uv}(-z)^4.
\end{equation}
To see the factor and sign, pair the positive and negative frequencies
in~\eqref{eq:conditional-fourier}. For every surviving
\(j+|\beta|\) the phase is real with square one.
The square of the resulting factor
\(2\sqrt{2/\pi}\Delta\), followed by \(\pi/2\), is
\(4\Delta^2\). The fourth power in~\eqref{eq:finite-head} sums over
all four auxiliary indices, while \(-z\) supplies their reflection.
A triangular implementation evaluates \(F(F+1)/2=295296\) pairs,
counting off-diagonal pairs twice.

\subsection{The zero Fourier node}

The zero node has two contributions. For odd \(j,k\ge1\), set
\begin{align*}
 \alpha_j&=2\Delta\sum_{u>0}\widetilde V_j(u)\widetilde p_0(u)^4,
 &z_j&=\Delta\widetilde V_j(0),\\
 \eta_k&=2\Delta\sum_{u>0}\widetilde V_0(u)
                             \widetilde p_k(u)\widetilde p_0(u)^3,
 &t_k&=\Delta\phi(0)\left(\ell\mathbf1_{k=1}
     +\sum_{l=1}^6\frac{b_l(3l/4)^k e^{-(3l/4)^2/2}}{\sqrt{k!}}\right).
\end{align*}
The complete correction to~\eqref{eq:finite-head}, restricted to degree
\(403\), is
\begin{equation}\label{eq:zero-correction}
 \sum_{j\text{ odd}}(2\alpha_jz_j+z_j^2)P(z)^j
 -4\sum_{k\text{ odd}}(2\eta_kt_k+t_k^2)z^k.
\end{equation}
For positive primary index, \(p_k(0)=\mathbf1_{k=0}\) leaves only the
first sum. For primary index zero, the principal-value quotient has a
nonzero zero limit exactly when one auxiliary index is positive.
Its derivative is
\(\phi(0)p'_k(0)\), as displayed in \(t_k\). There are four possible
positions, each reflected, giving the factor \(-4\).
Both cross terms and squares are therefore present.

\subsection{Analytic error and certified objective}

Appendix~\ref{app:head-errors} bounds the phase truncation, all Poisson
aliases, the omitted frequencies, the spatial lattice, its omitted nodes,
and the residual phase allowance. The strip estimates use the
exponential trapezoidal-rule principle~\cite{TrefethenWeideman2014}.
The fixed parameters are
\begin{equation}\label{eq:head-parameters}
 N=403,\qquad M=31,\qquad F=768,\qquad
 \Delta=2\pi/40,\qquad h_x=1/256,\qquad A_x=20.
\end{equation}
The certificate specifies all \(404\) rational contour choices.
If \(\delta\) bounds the Euclidean error of the retained Hermite
array, then Bessel's inequality and Cauchy--Schwarz give
\[
 \sum_{j,\beta}|A_{j,\beta}^2-\widetilde A_{j,\beta}^2|
 \le\delta(2+\delta).
\]
Composition by \(P\) and \(-z\), and restriction of degree, contract
the coefficient absolute norm. Thus the complete analytic coefficient
error is bounded by the stored rational allowance \(E\), with
\begin{equation}\label{eq:head-error-value}
 \frac\pi2\delta(2+\delta)\le E
              <1.597802294094\times10^{-10}.
\end{equation}

Every finite operation in~\eqref{eq:finite-head} and
\eqref{eq:zero-correction} is evaluated by outward-rounded ball
arithmetic. Let \([l_n,u_n]\) be the resulting coefficient intervals.
The certificate contains all \(404\) intervals, including zero and
even indices. Define
\begin{equation}\label{eq:head-objective}
 g=l_1-\sum_{n=3,5,\ldots,403}\max\{|l_n|,|u_n|\}
                                      -E.
\end{equation}
The coefficient objective is \(1\)-Lipschitz in the coefficient
absolute norm, so the total allowance is subtracted once.
Exact rational aggregation gives
\begin{equation}\label{eq:head-certified}
 c_1-\sum_{n=3,5,\ldots,403}|c_n|\ge g
 >\frac{88263349}{10^8}.
\end{equation}
For orientation, the lower endpoint is
\(g=0.8826334911000137\ldots\), and
the finite linear coefficient is approximately \(0.8826383020122081\).
The inequalities in~\eqref{eq:head-certified} use the exact rational
endpoints stored in the certificate.

\section{A bound for the entire infinite tail}\label{sec:tail}

Let \(D=z\,d/dz\). We prove a bound on \(\|D^4H\|_2\), where
\(L^2\) on the unit circle uses normalized arc measure.
The calculation combines nonnegative Hermite energy series with complex
Gaussian density bounds. Primary smoothing supplies the boundary
regularity needed for Parseval.

\subsection{Primary smoothing and mixed energies}

Write \(L(x)=U(S(x))\). For \(0<r<1\), smooth the primary variable:
\begin{equation}\label{eq:smoothing}
 G_r(w,x)=\E_Z f(\sqrt r\,w+\sqrt{1-r}\,Z,x)
 =\erf\left(\frac{\sqrt r\,w+L(x)}{\sqrt{2(1-r)}}\right).
\end{equation}
Its Hermite coefficient at \((j,\beta)\) is
\(r^{j/2}A_{j,\beta}\). Every derivative of positive fixed order is a
finite sum of bounded derivatives of \(L\), multiplied by a polynomial
in \(\sqrt r\,w+L(x)\) and a Gaussian exponential. It is therefore in
Gaussian \(L^2\).

For \(p,s\ge0\), \(p+s\ge1\), define
\begin{equation}\label{eq:mixed-energy}
 S_{p,s}(r)=\sum_{j,\beta}(j)_p(|\beta|)_s
                         r^{j-p}A_{j,\beta}^2,
\end{equation}
where \((n)_k=n(n-1)\cdots(n-k+1)\), with value zero for \(n<k\).
Terms with \(j<p\) are omitted, and the nonnegative power series
defines \(S_{p,s}(0)\) by continuity.
The Gaussian derivative identity gives
\begin{equation}\label{eq:sobolev-identity}
 \sum_{|\alpha|=s}\frac{s!}{\alpha!}
       \|\partial_w^p\partial_x^\alpha G_r\|_2^2
 =r^pS_{p,s}(r)<\infty.
\end{equation}
Indeed, differentiation lowers a normalized Hermite polynomial by one
with factor equal to the square root of its degree. The multinomial
identity
\(\sum_{|\alpha|=s}(s!/\alpha!)(\beta)_\alpha=(|\beta|)_s\)
then gives~\eqref{eq:sobolev-identity}.

Choosing \(q_0<r<1\) shows that every specified partial derivative of
\(\mathcal H\) converges absolutely and uniformly on
\(|a|\le q_0\), \(|b|\le1\). For example,
\begin{equation}\label{eq:energy-partial-bound}
 |\partial_a^p\partial_b^s\mathcal H(a,b)|
 \le\frac\pi2 S_{p,s}(|a|),\qquad |a|<1,\quad |b|\le1.
\end{equation}
The same domination at a larger radius justifies this inequality at
\(|b|=1\). In particular, \(H\) has four continuous angular
derivatives on the unit circle.

For the numerical evaluation of~\eqref{eq:mixed-energy}, put
\(v=1-r\), \(a_r=1+r\), and define
\[
 J_1(t)=1,\quad J_2(t)=-t/v,\quad
 J_3(t)=t^2/v^2-1/v,\quad J_4(t)=-t^3/v^3+3t/v^2.
\]
The \(m\)-th derivative of \(\erf(t/\sqrt{2v})\) is
\(\sqrt{2/\pi}\,v^{-1/2}e^{-t^2/(2v)}J_m(t)\).
Let \(\mathcal B_{\alpha,k}(x)\) be the multivariate Bell polynomial
determined by
\[
 \partial_x^\alpha h(L(x))
       =\sum_k h^{(k)}(L(x))\mathcal B_{\alpha,k}(x),
 \qquad \mathcal B_{0,0}=1.
\]
For \(1\le p+|\alpha|\le4\), it follows that
\begin{equation}\label{eq:smoothed-derivative}
 \partial_w^p\partial_x^\alpha G_r
 =r^{p/2}\sqrt{2/\pi}\,v^{-1/2}e^{-t^2/(2v)}
          \sum_kJ_{p+k}(t)\mathcal B_{\alpha,k}(x),
 \qquad t=\sqrt r\,w+L(x).
\end{equation}
Conditional on \(x\), the variable \(t\) is normal with mean \(L\)
and variance \(r\). Multiplication of its density by \(e^{-t^2/v}\)
and square completion give mass
\(\sqrt{v/a_r}\,e^{-L^2/a_r}\), with tilted normal variable
\[
 T\sim N\left(\frac{vL}{a_r},\frac{rv}{a_r}\right).
\]
Consequently \(S_{p,s}(r)\) is obtained from
\begin{equation}\label{eq:conditional-energy}
 \frac{2}{\pi\sqrt{1-r^2}}
 \sum_{|\alpha|=s}\frac{s!}{\alpha!}
 \E_X\left[e^{-L(X)^2/a_r}
    \sum_{k,l}\mathcal B_{\alpha,k}(X)\mathcal B_{\alpha,l}(X)
                       \E_T[J_{p+k}(T)J_{p+l}(T)]\right].
\end{equation}
The factor \(r^p\) from~\eqref{eq:smoothed-derivative} cancels the
one in~\eqref{eq:sobolev-identity}. Formula~\eqref{eq:conditional-energy}
retains the signs of every cross term.

Expanding \(L=U\circ S\) in this formula produces a finite rational
linear combination of integrals of the form
\begin{equation}\label{eq:energy-integral-form}
 \E\left[e^{-U(S)^2/a_r}U(S)^h
       \prod_{j=1}^4U^{(j)}(S)^{o_j}
       \prod_{i=1}^4\prod_{j=1}^4B^{(j)}(X_i)^{\lambda_{i,j}}\right].
\end{equation}
Appendix~\ref{app:energy-errors} gives their Fourier quadrature and all
analytic allowances. The calculation encloses all \(14\) pairs
\((p,s)\) with \(1\le p+s\le4\) at each radius
\(81/100,83/100,\ldots,97/100\).
There are \(25\) distinct local transforms and \(109\) scalar profiles
at each radius, giving \(981\) signed scalar calculations and
\(126\) energy endpoints. The derivative polynomials are also checked
by direct differentiation over all ordered auxiliary coordinate tuples.

Every power series \(S_{p,s}(r)\) has nonnegative coefficients and
nonnegative integer exponents. It is increasing and convex.
If certified upper bounds at consecutive radii \(r_0<r_1\) are
\(E_0,E_1\), then
\begin{equation}\label{eq:energy-interpolation}
 S_{p,s}(r)\le
 \frac{r_1-r}{r_1-r_0}E_0+
 \frac{r-r_0}{r_1-r_0}E_1\qquad(r_0\le r\le r_1).
\end{equation}
Below the smallest radius, its endpoint bound applies by monotonicity.
Monotonicity also permits replacing an endpoint bound by a smaller bound
from a later radius.

\subsection{Complex Gaussian density bounds}

A complementary estimate is effective when the correlations are away
from the real endpoints. For \(z\in\mathbb C\) with
\(|\operatorname{Re}z|<1\), the analytically continued Gaussian pair
density is
\[
 \rho_z(x,y)=\frac1{2\pi\sqrt{1-z^2}}
       \exp\left(-\frac{x^2-2zxy+y^2}{2(1-z^2)}\right),
\]
with the branch obtained by continuation from zero.
In coordinates \(u=(x+y)/\sqrt2\), \(v=(x-y)/\sqrt2\), its
absolute value is a positive constant times the product of real Gaussian
densities with variances
\begin{equation}\label{eq:density-mass}
 \sigma_\pm(z)=\frac{|1\pm z|^2}{1\pm\operatorname{Re}z},
 \qquad
 d_0(z)=\iint|\rho_z(x,y)|\,dx\,dy
      =\sqrt{\frac{|1-z^2|}{1-(\operatorname{Re}z)^2}}.
\end{equation}
The real-part condition ensures integrability and fixes the square-root
branch: \(1-z^2\) stays away from the nonpositive real axis.

The following finite recurrence bounds all density derivatives through
order four. With formal nonnegative variables \(U,V\), put
\begin{align*}
 L_k(U,V;z)={}&\frac{(k-1)!}{2}
          \left(|1+z|^{-k}+|1-z|^{-k}\right)\\
 &+\frac{k!}{2}\left(U|1+z|^{-k-1}
                         +V|1-z|^{-k-1}\right),\\
 T_0(U,V;z)={}&1,\\
 T_j(U,V;z)={}&\sum_{k=1}^j\binom{j-1}{k-1}
                    T_{j-k}(U,V;z)L_k(U,V;z).
\end{align*}
The polynomial \(L_k(u^2,v^2;z)\) bounds the absolute value of the
\(k\)-th derivative of \(\log\rho_z\). The Bell recurrence thus
bounds \(|\partial_z^j\rho_z|\) by
\(|\rho_z|T_j(u^2,v^2;z)\).
Define the positive linear functional
\[
 \mathcal M_z(U^aV^b)=
 \frac{(2a)!}{2^aa!}\sigma_+(z)^a
 \frac{(2b)!}{2^bb!}\sigma_-(z)^b,
 \qquad D_j(z)=d_0(z)\mathcal M_z(T_j).
\]
It follows that \(\iint|\partial_z^j\rho_z|\le D_j(z)\).
Four auxiliary pairs share the same correlation, so their joint
order-\(s\) bound is
\begin{equation}\label{eq:aux-density-bound}
 D_s^{[4]}(b)=s!\,[u^s]
                 \left(\sum_{j=0}^{4}\frac{D_j(b)}{j!}u^j\right)^4.
\end{equation}
Integrating the product of the five pair densities against the two
bounded sign functions gives
\begin{equation}\label{eq:density-partial-bound}
 |\partial_a^p\partial_b^s\mathcal H(a,b)|
 \le\frac\pi2D_p(a)D_s^{[4]}(b),
 \quad |a|\le q_0,\quad |b|\le1,\quad
 |\operatorname{Re}a|,|\operatorname{Re}b|<1.
\end{equation}
The integral and Hermite definitions agree first in the open bidisk by
Gaussian expansion and then on the stated boundary arcs by analytic
continuation and the regularity proved above.

\subsection{Circle enclosure and coefficient tail}

For fixed \(|z|=1\), set \(a(u)=P(ze^u)\) and \(b(u)=-ze^u\).
The full fourth-order chain rule is
\begin{equation}\label{eq:chain-rule}
 D^4H(z)=\sum_{1\le p+s\le4}
       \frac{4!}{p!s!}[u^4]
       (a(u)-a(0))^p(b(u)-b(0))^s\,
       \partial_a^p\partial_b^s\mathcal H(a(0),b(0)).
\end{equation}
Thus every required coefficient is a polynomial in the first four
Euler derivatives of \(P\) and of \(-z\).
On each interval, outward interval evaluation bounds these coefficients
and each applicable partial derivative using
\eqref{eq:energy-partial-bound}--\eqref{eq:energy-interpolation} or
\eqref{eq:density-partial-bound}. When both are available, the smaller
upper bound applies. At the real auxiliary endpoints, the mixed-energy
bound supplies the enclosure.

The program uses \(+z\) for the auxiliary jet. Replacing it by the
physical \(-z\) multiplies a fixed auxiliary-order coefficient by
\((-1)^s\); its absolute value is unchanged. Likewise,
\(D_j(-z)=D_j(z)\) by interchanging the two Gaussian coordinates.
The computed majorant therefore bounds the kernel
in~\eqref{eq:scalar-kernel}.

The first quadrant is partitioned into the \(2048\) closed intervals
\begin{equation}\label{eq:circle-panels}
 I_i=\left[\frac{i\pi}{4096},\frac{(i+1)\pi}{4096}\right],
 \qquad 0\le i<2048.
\end{equation}
For every interval the calculation encloses all \(14\) mixed derivative
terms in~\eqref{eq:chain-rule}. Write \(M_i\) for its upper bound on
\(|D^4H(e^{i\theta})|\) throughout \(I_i\).
The real odd coefficients give fourfold symmetry, so normalized circle
integration yields
\begin{equation}\label{eq:circle-bound}
 \|D^4H\|_2^2\le\frac1{2048}\sum_{i=0}^{2047}M_i^2
       \le C_4^2,
 \qquad C_4=\frac{1006651833}{2048}.
\end{equation}
The second comparison is performed with rational upper endpoints of all
panel enclosures.

\begin{theorem}[Complete tail bound]\label{thm:tail}
The scalar kernel satisfies
\[
 \sum_{\substack{n>403\\n\text{ odd}}}|c_n|
 \le\frac{1006651833/2048}{\sqrt{14\cdot403^7}}
 <\frac1{10000}.
\]
\end{theorem}

\begin{proof}
Boundary regularity and Parseval give
\[
 \sum_{n\ge1}n^8|c_n|^2=\|D^4H\|_2^2\le C_4^2.
\]
For odd \(N\), the positive odd indices greater than \(N\) are
\(N+2,N+4,\ldots\). Since \(x^{-8}\) decreases,
\[
 \sum_{\substack{n>N\\n\text{ odd}}}n^{-8}
 \le\frac12\int_N^\infty x^{-8}\,dx=\frac1{14N^7}.
\]
Cauchy--Schwarz gives the first bound with \(N=403\).
The strict numerical comparison follows from the exact identity
\[
 14\cdot403^7-10^8C_4^2
 =\frac{150131951650291275887}{16384}>0.
\]
\end{proof}

\section{Exact conclusion and computational record}\label{sec:conclusion}

Combining~\eqref{eq:head-certified} with the tail theorem gives
\begin{equation}\label{eq:gamma-final}
 \Gamma=c_1-\sum_{n\ge3}|c_n|>
 \Gamma_0:=\frac{88253349}{10^8}.
\end{equation}
For a rational bound on \(\pi\), define
\[
 T_N(x)=\sum_{k=0}^{N-1}\frac{(-1)^kx^{2k+1}}{2k+1}.
\]
For \(0<x<1\), the alternating-series bound gives
\(T_{2r}(x)<\arctan x<T_{2r+1}(x)\).
Machin's identity therefore gives the rational upper bound
\begin{equation}\label{eq:machin}
 \pi<\pi_+:=16\left(T_{100}(1/5)
                  +\frac{1}{201\cdot5^{201}}\right)
                     -4T_{30}(1/239)
       <\frac{3141592653589794}{10^{15}}.
\end{equation}
The last comparison is an integer calculation. Moreover,
\begin{equation}\label{eq:exact-final-comparison}
 2\Gamma_0\frac{44496791}{25000000}
 -\frac{3141592653589794}{10^{15}}
 =\frac{17031633}{2500000000000000}>0.
\end{equation}
Proposition~\ref{prop:transfer} now yields
\[
 \KG\le\frac\pi{2\Gamma}
       <\frac\pi{2\Gamma_0}<\frac{44496791}{25000000},
\]
which proves Theorem~\ref{thm:main}.

Table~\ref{tab:coverage} records the finite coverage of the computation.
All transcendental quantities are enclosed by ball arithmetic based on
Arb~\cite{Johansson2017}. The production run used Python \(3.12\),
python-flint \(0.8.0\), and FLINT \(3.3.1\).
Head primitives and pairs used \(512\)-bit precision; analytic head
error and final assembly used \(640\) bits; mixed energies used
\(256\) bits; and circle bounds used \(384\) bits.
The final coefficient aggregation and comparisons use exact rational
arithmetic.

\begin{table}[htbp]
\centering
\caption{Complete finite coverage and exact output thresholds.}
\label{tab:coverage}
\small
\begin{tabular}{@{}p{.59\linewidth}r@{}}
\toprule
Component & Coverage or certified bound\\
\midrule
Scalar coefficients, degrees \(0\) through \(403\) & \(404\) intervals\\
Positive Fourier nodes & \(768\)\\
Hermite characteristics and profiles at each node & \(404+404\)\\
Triangular frequency pairs & \(295296\)\\
Ordered pair terms, including multiplicity & \(589824\)\\
Local transforms for mixed energies & \(25\)\\
Signed scalar integrals & \(981\)\\
Mixed-energy orders at each of nine radii & \(14\)\\
First-quadrant circle panels & \(2048\)\\
Mixed derivative bounds across the panels & \(28672\)\\
Finite head objective & \(>88263349/10^8\)\\
Fourth angular-derivative norm & \(\le1006651833/2048\)\\
Infinite coefficient tail & \(<1/10000\)\\
\bottomrule
\end{tabular}
\end{table}

Additional checks used different algebraic organizations of the finite
expressions. A direct multiindex calculation reproduced the low-degree
head through degree \(5\); an exact synthetic calculation compared the
multiindex and pair-product formulas; and direct ordered-coordinate
differentiation matched all \(126\) mixed-energy polynomials.
These checks accompany the full source and analytic review of the
certificate pipeline.

\section{Reproducing the certificate}\label{sec:reproduction}

The ancillary source package contains the rational input
\texttt{physical.json}, the head integration plan, all retained scalar
coefficient intervals, the \(2048\) circle bounds, and executable
verification and regeneration commands.
From its source directory, the retained certificate is checked with
\begin{verbatim}
python -B verify.py
\end{verbatim}
This standard-library command checks the file hashes, preprocessing norm,
coefficient objective, circle aggregation, infinite-tail comparison,
Machin bound and final rational inequality.

For a full regeneration, create a Python \(3.12\) virtual environment,
install the pinned requirements, and run
\begin{verbatim}
python -m pip install -r requirements.txt
python -B reproduce.py full --workdir ../full-run --workers 4
\end{verbatim}
The work directory must be new. The driver computes every component in
Table~\ref{tab:coverage}, assembles the new certificate, and writes its
verification result. It downloads twelve arithmetic modules from the
pinned source commit and verifies their SHA-256 hashes. A verified cache
supports offline execution:
\begin{verbatim}
python -B reproduce.py bootstrap --cache-dir ../dependency-cache
\end{verbatim}
Then append \texttt{--cache-dir ../dependency-cache --offline} to the
full-regeneration command, using a new work directory.
Python assertions must remain enabled. A shorter
\texttt{reproduce.py smoke} run checks representative primitive and pair
calculations, the complete analytic error plans, and all exact
mixed-energy polynomials. The package's \texttt{USAGE.txt} gives the
command details and output locations.

The construction and arithmetic kernels are pinned to Ren's repository,
commit \texttt{af6990520d97d232433a55c93a5e9b6bc6b56fa7}~\cite{Ren2026}.
The full computation reported here regenerates the head, mixed energies
and circle estimate from those rational inputs. The package records the
source hashes and the checks performed on the resulting data.

\paragraph{Acknowledgement.}
OpenAI Codex assisted with mathematical analysis, implementation,
verification and manuscript preparation.

\appendix
\section{Analytic errors in the coefficient head}\label{app:head-errors}

We give the error bounds for the finite Fourier calculation, following
the analytic route of~\cite{Ren2026}. All Hermite polynomials are normalized
as $\psi_k=\He_k/\sqrt{k!}$, and $\phi(x)=e^{-x^2/2}/\sqrt{2\pi}$.
Finite arithmetic is performed with outward-rounded intervals; the
allowances below account for analytic approximation.

\subsection{Phase expansion and Gaussian envelopes}

Put $c=\sqrt{2/\pi}$ and $A_*=\sum_l|a_l|$.
For $j\ge1$, the conditional coefficient has Fourier representation
\[
 q_j(s)=c i^{j-1}\int_{\R}
       \frac{\phi(u)u^{j-1}}{\sqrt{j!}}e^{ius}\,du.
\]
For $j=0$, the integrand is $\phi(u)/u$ and the integral is its
symmetric principal value. Differentiation and the value at zero give
$q_0(s)=\erf(s/\sqrt2)$.
Let
\[
 q_j^{[M]}(s)=\sum_{m=0}^{M}\frac{A(s)^m}{m!}q_j^{(m)}(s),
 \qquad
 \widetilde V_j=\text{its Fourier weight in the preceding convention}.
\]
Taylor's formula for $e^{iuA(s)}$, whose real-phase derivatives have
modulus one, gives
\begin{equation}\label{eq:head-phase}
 \sup_{s\in\R}|q_j(U(s))-q_j^{[M]}(s)|\le e_j,
 \qquad
 e_j=\frac{cA_*^{M+1}\E|Z|^{j+M}}{(M+1)!\sqrt{j!}},
 \qquad
 \delta_{\rm ph}=\left(\sum_{j=0}^{N}e_j^2\right)^{1/2}.
\end{equation}
Here $Z$ is standard normal. For $j=0$, the remainder cancels the
factor $1/u$. Bessel's inequality in the auxiliary variables proves the
array bound $\delta_{\rm ph}$, with no multiindex counting factor.

For $0<t<1$, set $a_t=1+t$ and $C_t=(1-t^2)^{-1/4}$. The diagonal
complex Mehler identity, followed by retaining its $n$th nonnegative term,
gives
\begin{equation}\label{eq:head-mehler}
 \frac{|\He_n(x+iy)|}{\sqrt{n!}}e^{-(x^2-y^2)/2}
 \le C_t t^{-n/2}
       \exp\left(-\frac{x^2}{2a_t}+\frac{y^2}{2(1-t)}\right).
\end{equation}
Indeed, the squared Hermite sum before multiplication by
$e^{-x^2+y^2}$ is
$(1-t^2)^{-1/2}\exp(tx^2/(1+t)+ty^2/(1-t))$.
Define
\[
 A_y=\sum_l|a_l|\cosh(|\nu_l|y),\qquad
 A'_y=\sum_l|a_l\nu_l|\cosh(|\nu_l|y),
\]
\begin{align*}
 K_j(t,y)&=cC_t\sum_{m=0}^{M}\frac{A_y^m}{m!}
        \sqrt{\frac{(j+m-1)!}{j!}}t^{-(j+m-1)/2}
        \quad(j\ge1),\\
 K'_0(t,y)&=cC_t\left\{
    \sum_{m=0}^{M}\frac{A_y^m\sqrt{m!}}{m!}t^{-m/2}
    +A'_y\sum_{m=1}^{M}\frac{A_y^{m-1}\sqrt{(m-1)!}}{(m-1)!}
                 t^{-(m-1)/2}\right\}.
\end{align*}
Equation~\eqref{eq:head-mehler} bounds $|q_j^{[M]}(x+iy)|$ by
$K_j(t,y)e^{-x^2/(2a_t)+y^2/(2(1-t))}$, and bounds
$|(q_0^{[M]})'(x+iy)|$ by the same expression with $K'_0$.

We also use the following strip quadrature rule. If an analytic,
sufficiently decaying $v$ satisfies
$\int_{\R}|v(x\pm iy)|\,dx\le J_\pm$, then
\begin{equation}\label{eq:strip-rule}
 \left|h\sum_{k\in\mathbb Z}v(kh)-\int_{\R}v(x)\,dx\right|
 \le\frac{J_++J_-}{e^{2\pi y/h}-1}.
\end{equation}
Contour shifting bounds the Fourier transform at $2\pi n/h$ by
$J_+e^{-2\pi ny/h}$ or $J_-e^{-2\pi |n|y/h}$ according to its sign;
Poisson summation and two geometric series prove the rule. The integrands
to which we apply it have Gaussian decay on each fixed horizontal strip.

\subsection{Fourier lattice and frequency cutoff}

Write $S=G+R$, where
\[
 \sigma^2=4\ell^2,\qquad G\sim N(0,\sigma^2),\qquad
 |R|\le D:=4\sum_l|b_l|.
\]
The remainder $R$ may depend on $G$.
For the alias estimate choose $0<\beta<1$, and put
\[
 C=\frac{\beta D^2}{2a_t(1-\beta)},\quad
 \kappa=\frac{\beta}{a_t+2\beta\sigma^2},\quad
 c_L=\frac{\kappa L^2}{2},\quad
 \mu=\sqrt{\frac{\beta\sigma^2}{2a_t}}.
\]
If $2e^{-3c_L}<1$, valid $L^2$ allowances for the complete Fourier
lattice are
\begin{equation}\label{eq:head-alias}
 \alpha_j=\frac{2K_j(t,0)e^{C-c_L}}{1-e^{-3c_L}}\quad(j\ge1),
 \qquad
 \alpha_0=\frac{2K'_0(t,0)(D+\sigma+\mu L)e^{C-c_L}}
                    {1-2e^{-3c_L}}.
\end{equation}
For $j\ge1$, the identity
$(|x|-D)_+^2\ge\beta x^2-\beta D^2/(1-\beta)$ and Gaussian
square completion imply
$\|q_j^{[M]}(S+mL)\|_2\le K_j(t,0)e^{C-c_Lm^2}$.
Summing the nonzero translates proves the first bound.

For $j=0$, apply Poisson summation to $(q_0^{[M]})'$, whose integral
equals $2$, and integrate from zero to $S$. Oddness fixes the integration
constant. The error is
\[
 \sum_{m\ne0}\int_0^1 S(q_0^{[M]})'(vS+mL)\,dv.
\]
In the Gaussian square completion at $vG+mL$, the tilted variance of
$G$ is at most $\sigma^2$, and the absolute tilted mean is at most
$\mu|m|L$, uniformly for $0\le v\le1$.
Thus each summand has norm at most
$K'_0(t,0)(D+\sigma+\mu|m|L)e^{C-c_Lm^2}$.
The ratio of successive positive-index majorants is at most
$2e^{-3c_L}$, proving the second bound. The zero Fourier coefficient
of the derivative supplies the linear term $2S/L$ used in the
zero-node correction.

For the omitted frequencies, choose $y>0$ and a new $t\in(0,1)$ if
desired, and set
\[
 C_j(y)=\frac{\sqrt{a_t}}2K_j(t,y)e^{y^2/(2(1-t))},\qquad
 C'_0(y)=\frac{\sqrt{a_t}}2K'_0(t,y)e^{y^2/(2(1-t))}.
\]
Contour shifting in the inverse Fourier formula gives
$|\widetilde V_j(u)|\le C_j(y)e^{-y|u|}$ for $j\ge1$.
Shifting the derivative for $j=0$ gives
$|\widetilde V_0(u)|\le C'_0(y)e^{-y|u|}/|u|$ for $u\ne0$.
With $\Delta=2\pi/L$, the resulting uniform errors are
\begin{equation}\label{eq:head-frequency}
 \begin{split}
 \varrho_j&=\frac{2c\Delta C_j(y)e^{-y(F+1)\Delta}}
                         {1-e^{-y\Delta}}\quad(j\ge1),\\
 \varrho_0&=\frac{2c\Delta C'_0(y)e^{-y(F+1)\Delta}}
                  {(F+1)\Delta(1-e^{-y\Delta})}.
 \end{split}
\end{equation}
Each degree may use its own certified rational pair $(t,y)$.

\subsection{Spatial transforms and propagation}

Consider $h_k(u)=\E[\psi_k(X)e^{iuB(X)}]$, for $|u|\le T=F\Delta$.
Use spacing $h_x$, cutoff $A_x=K_xh_x$, and strip height $y_x>0$.
The translation identity for Hermite polynomials yields
\[
 D_k(y_x):=\E|\psi_k(X+iy_x)|^2
       =\sum_{r=0}^{k}\binom{k}{r}\frac{y_x^{2r}}{r!},\qquad
 B_{\rm im}=|\ell|y_x+\sum_l|b_l|\sinh(|\omega_l|y_x).
\]
Cauchy--Schwarz on the shifted line and~\eqref{eq:strip-rule} bound
the infinite spatial lattice error by $e_{x,k}$ below. Applying
\eqref{eq:head-mehler} on the real line bounds the discarded nodes by
$t_{x,k}$, using the decreasing Gaussian envelope:
\begin{equation}\label{eq:head-spatial}
 \begin{split}
 e_{x,k}&=\frac{2\sqrt{D_k(y_x)}e^{y_x^2/2+TB_{\rm im}}}
                          {e^{2\pi y_x/h_x}-1},\\
 t_{x,k}&=C_t t^{-k/2}\sqrt{a_t}\,
                     \erfc\left(\frac{A_x}{\sqrt{2a_t}}\right).
 \end{split}
\end{equation}
For a binned transform, require the whole interval of each residual
$B(kh_x)-m_kL/M_b$ to lie in $[-\delta_b,\delta_b]$.
Expanding its real phase to degree $J_b$ introduces the further allowance
\begin{equation}\label{eq:head-bins}
 \begin{split}
 g_x&=\frac{2e^{y_x^2/2}}{e^{2\pi y_x/h_x}-1},\qquad
 \mathcal M_k=\sqrt{(1+g_x)(1+g_xD_k(y_x))},\\
 b_{x,k}&=\mathcal M_k\frac{(T\delta_b)^{J_b+1}}{(J_b+1)!}.
 \end{split}
\end{equation}
To obtain $\mathcal M_k$, apply the strip rule to $\phi$ and
$\phi\psi_k^2$, then use Cauchy--Schwarz for the full lattice sum
of $\phi|\psi_k|$. The finite direct sum used here retains this
nonnegative bin allowance in its error budget.

Set $\eta=\sqrt{N+1}(e_{x,N}+t_{x,N}+b_{x,N})$; the positive
expressions in~\eqref{eq:head-spatial}--\eqref{eq:head-bins} increase
with $k$. Bessel's inequality gives $\|(h_k(u))_{k=0}^{N}\|_2\le1$.
The error of its fourfold tensor is therefore bounded by
$\tau_\eta=(1+\eta)^4-1$.
The resulting error for the $j$th conditional coefficient array is at most
\[
 b_j=2c\Delta\tau_\eta\sum_{n=1}^{F}|\widetilde V_j(n\Delta)|.
\]
The evaluation uses the convenient uniform bounds
\[
 b_0\le c\sqrt{a_t}K'_0(t,0)\tau_\eta\sum_{n=1}^{F}\frac1n,
 \qquad
 b_j\le c\Delta F\sqrt{a_t}K_j(t,0)\tau_\eta\quad(j\ge1).
\]
The zero-node values in~\eqref{eq:zero-correction} are evaluated directly
with interval arithmetic. Combining the preceding estimates gives
\begin{equation}\label{eq:head-total-error}
 \delta=\delta_{\rm ph}+
       \left(\sum_{j=0}^{N}(\alpha_j+\varrho_j+b_j)^2\right)^{1/2},
 \qquad E_{\rm head}=\frac\pi2\delta(2+\delta).
\end{equation}
Indeed, if $A$ is the exact Hermite array and $\widetilde A$ its finite
approximation, then $\|A\|_2\le1$, $\|A-\widetilde A\|_2\le\delta$,
and
$\sum|A_{j,\beta}^2-\widetilde A_{j,\beta}^2|
\le\|A-\widetilde A\|_2\|A+\widetilde A\|_2\le\delta(2+\delta)$.
Since $\|P\|_1\le1$, multiplication by
$P(z)^j(-z)^{|\beta|}$ and restriction to degrees at most $N$
preserve this absolute-sum bound after the factor $\pi/2$.

The calculation uses
\[
 \begin{gathered}
 N=403,\quad M=31,\quad L=40,\quad F=768,\quad
 h_x=1/256,\quad A_x=20,\quad y_x=1/8,\\
 t=99/100,\quad\beta=1/3,\quad
 M_b=2048,\quad\delta_b=1/100,\quad J_b=32.
 \end{gathered}
\]
The $404$ rational contour choices and their individual enclosures are
included in the certificate. Outward evaluation of
\eqref{eq:head-total-error} yields
$E_{\rm head}<1.597802295\times10^{-10}$.

\section{Analytic errors in the mixed energies}\label{app:energy-errors}

Each mixed-energy term from the construction of~\cite{Ren2026} has the
form
\[
 I=\E\left[g(S)\prod_{i=1}^{4}w_i(X_i)\right],\qquad
 g(s)=e^{-U(s)^2/a}U(s)^h\prod_{j\ge1}U^{(j)}(s)^{o_j},\qquad
 w_i(x)=\prod_{j\ge1}B^{(j)}(x)^{\lambda_{i,j}},
\]
where $a=1+r$, $h,o_j,\lambda_{i,j}$ are nonnegative integers,
and the products are finite. Put
\[
 \widehat g(t)=\frac1{2\pi}\int_{\R}g(s)e^{-its}\,ds,
 \qquad \varphi_i(t)=\E[w_i(X)e^{itB(X)}].
\]
Fourier inversion and independence give
$I=\int_{\R}\widehat g(t)\prod_i\varphi_i(t)\,dt$.
We approximate this integral by its finite lattice with
\begin{equation}\label{eq:energy-grid}
 \begin{gathered}
 L_0=64,\quad\Delta_0=\pi/32,\quad N_f=3072,\quad T=96\pi,\\
 h_x=1/512,\quad A_x=14,\quad h_s=1/128,\quad Z=16,\quad
 y_x=1/8,\quad y_s=y=2/5,\\
 \delta_b=1/250,\quad J=31,\quad M_0=8192.
 \end{gathered}
\end{equation}
The local transforms use a positive-sign discrete Fourier transform,
the scalar transform uses a negative-sign transform, and the outer sum
has factor $\Delta_0$. Conjugate symmetry counts zero frequency once
and every positive frequency as twice the real part.

\subsection{Local and scalar transforms}

For $v\ge0$ define the local derivative majorant
\[
 W_i(v)=\prod_{j\ge1}\left(
 |\ell|\mathbf1_{j=1}+\sum_l|b_l||\omega_l|^j
                               \cosh(|\omega_l|v)\right)^{\lambda_{i,j}},
 \qquad W_i=W_i(0).
\]
With $B_{\rm im}$ as in Appendix~\ref{app:head-errors}, the error in
$\varphi_i(t)$, uniformly for $|t|\le T$, is at most
\begin{equation}\label{eq:energy-local-error}
 e_i=\frac{2W_i(y_x)e^{y_x^2/2+TB_{\rm im}}}
                     {e^{2\pi y_x/h_x}-1}
       +W_i\erfc(A_x/\sqrt2)
       +W_i(1+g_x)\frac{(T\delta_b)^{J+1}}{(J+1)!},
 \quad
 g_x=\frac{2e^{y_x^2/2}}{e^{2\pi y_x/h_x}-1}.
\end{equation}
The three terms follow respectively from the strip quadrature rule
\eqref{eq:strip-rule}, the decreasing Gaussian tail, and Taylor's theorem
for the real residual phase. The last term uses the full Gaussian
lattice mass bound $1+g_x$.
At each retained node, an integer bin $m$ is chosen with the verified
interval condition $|B(x)-mh_s|\le\delta_b$. Since $M_0h_s=L_0$,
the bin phase at $n\Delta_0$ is exactly a discrete Fourier phase.

For the scalar transform, set
\begin{align*}
 A_v&=\sum_l|a_l|\cosh(|\nu_l|v),&
 B_v&=v+\sum_l|a_l|\sinh(|\nu_l|v),&D_v&=A_v+B_v,\\
 E_v&=\prod_{j\ge1}\left(
 \mathbf1_{j=1}+\sum_l|a_l||\nu_l|^j
                   \cosh(|\nu_l|v)\right)^{o_j}.
\end{align*}
For $s=x+iv$ one has
$|\operatorname{Re}U(s)-x|\le A_v$,
$|\operatorname{Im}U(s)|\le B_v$, and $|U(s)|\le|x|+D_v$.
It follows that $(2\pi)^{-1}\int_{\R}|g(x+iv)|\,dx\le C_v$, where
\begin{equation}\label{eq:energy-strip-constant}
 C_v=\frac{E_ve^{B_v^2/a}}\pi
 \left\{\frac{(A_v+D_v)^{h+1}-D_v^{h+1}}{h+1}
 +\sum_{k=0}^{h}\binom hk(A_v+D_v)^{h-k}
             \frac{a^{(k+1)/2}}2\Gamma\left(\frac{k+1}{2}\right)\right\}.
\end{equation}
To prove this formula, bound the exponential by
$e^{B_v^2/a}e^{-(|x|-A_v)_+^2/a}$, split the positive half-line at
$A_v$, and expand $(x+D_v)^h$ after translating $x-A_v$.
Contour shifting then gives
$|\widehat g(t)|\le C_ve^{-v|t|}$.

For the scalar spatial tail assume $Z>A_0$ and
$Z^2\ge A_0^2+ah/2$, as checked for each term. These conditions make
$e^{-(s-A_0)^2/a}(s+A_0)^h$ decreasing for $s\ge Z$.
Write $R_0=Z-A_0$ and define
\begin{align*}
 J_0(R_0)&=\frac{\sqrt{\pi a}}2\erfc(R_0/\sqrt a),&
 J_1(R_0)&=\frac a2e^{-R_0^2/a},\\
 J_k(R_0)&=\frac a2R_0^{k-1}e^{-R_0^2/a}
                    +\frac{a(k-1)}2J_{k-2}(R_0)\quad(k\ge2).
\end{align*}
Integration by parts identifies $J_k(R_0)$ with
$\int_{R_0}^{\infty}t^ke^{-t^2/a}\,dt$.
Thus the scalar transform error, uniformly for $|t|\le T$, is
\begin{equation}\label{eq:energy-scalar-error}
 e_g=\frac{2C_{y_s}e^{Ty_s}}{e^{2\pi y_s/h_s}-1}
       +\frac{E_0}\pi\sum_{k=0}^{h}\binom hk(2A_0)^{h-k}J_k(R_0).
\end{equation}
The first term is~\eqref{eq:strip-rule} applied to
$(2\pi)^{-1}g(s)e^{-its}$; the second bounds both omitted spatial
tails, including the Fourier normalization.

\subsection{Frequency cutoff and Poisson translates}

The omitted lattice frequencies contribute at most $W\varepsilon_f$,
where $W=\prod_iW_i$ and
\begin{equation}\label{eq:energy-frequency-error}
 \varepsilon_f=
 \frac{2\Delta_0 C_y e^{-y(N_f+1)\Delta_0}}{1-e^{-y\Delta_0}}.
\end{equation}
Poisson summation also introduces the nonzero translates
$\E[g(S+nL_0)\prod_iw_i(X_i)]$.
To bound their total, write $S=G+R$ as above and put
\[
 \sigma^2=4\ell^2,\quad D=4\sum_l|b_l|+A_0,\quad
 c_r=\frac a{a+2\beta\sigma^2},\quad
 \tau^2=\sigma^2c_r,\quad
 \kappa=\frac\beta{a+2\beta\sigma^2},\qquad\beta=\frac12.
\]
For each translate,
$U(S+nL_0)=G+nL_0+\eta$ with $|\eta|\le D$.
The inequality
$(x+\eta)^2\ge\beta x^2-\beta D^2/(1-\beta)$ and square completion
against the Gaussian law of $G$ bound the absolute expectation by
\[
 W E_0\sqrt{c_r}
   e^{\beta D^2/[a(1-\beta)]-\kappa n^2L_0^2}P_h(|n|L_0),
\]
where
\[
 P_h(v)=\sum_{k=0}^{h}\binom hk(D+c_rv)^{h-k}\tau^k\mu_k,
 \qquad \mu_k=\frac{2^{k/2}\Gamma((k+1)/2)}{\sqrt\pi}.
\]
Here $\mu_k$ is the $k$th absolute standard Gaussian moment.
The tilted variable $G+nL_0$ has mean $c_rnL_0$ and variance $\tau^2$.
The polynomial $P_h$ has nonnegative coefficients and degree at most
$h$, so $P_h((n+1)L_0)\le2^hP_h(nL_0)$ for $n\ge1$.
For every term the calculation verifies
$\rho=2^he^{-3\kappa L_0^2}<1$. Summing both signs gives the
translate allowance $W\varepsilon_t$, with
\begin{equation}\label{eq:energy-translate-error}
 \varepsilon_t=
 \frac{2E_0\sqrt{c_r}
    e^{\beta D^2/[a(1-\beta)]-\kappa L_0^2}P_h(L_0)}{1-\rho}.
\end{equation}

\subsection{Combination of the allowances}

At a retained frequency, $|\widehat g|\le C_0$ and
$|\varphi_i|\le W_i$. The complete product error is bounded by
\[
 \mathcal E=(C_0+e_g)\prod_{i=1}^{4}(W_i+e_i)-C_0\prod_{i=1}^{4}W_i.
\]
This positive expression can be evaluated by starting with
$b=C_0$, $e=e_g$ and, for each local factor, replacing
\[
 (b,e)\longmapsto\bigl(bW_i,\ e(W_i+e_i)+be_i\bigr).
\]
The formula retains every cross term. Combining all retained frequencies,
the frequency tail and the Poisson translates proves
\begin{equation}\label{eq:energy-total-error}
 |I-I_{\rm finite}|
 \le\Delta_0(2N_f+1)\mathcal E
                         +W(\varepsilon_f+\varepsilon_t).
\end{equation}
For a mixed energy, multiply this bound by the absolute value of each
exact signed coefficient, sum over the compiled terms, and multiply by
the common factor $2/(\pi\sqrt{1-r^2})$.
Its finite Fourier sum is evaluated with interval arithmetic.
Equations~\eqref{eq:energy-local-error}--\eqref{eq:energy-total-error}
then enclose the complete mixed energy at each of the nine radii
$r=0.81,0.83,\ldots,0.97$ used in the circle estimate.

\section{Exact rational parameters}\label{app:parameters}

The following tables specify all coefficients in
\eqref{eq:boundary-local}--\eqref{eq:preprocessing}.
The ancillary file \texttt{physical.json} stores the same rational data.
Its SHA-256 digest is
\begin{center}\footnotesize\ttfamily
3077b5a8011868ff8a9200590bebb73cb41ad3ba10b09eb37860d5074c3b1653
\end{center}
The auxiliary preprocessing is \(Q(z)=z\); the collective amplitude
in the source model is zero. The local frequencies are \(3k/4\),
\(1\le k\le6\), and the outer frequencies are \(2k\),
\(1\le k\le4\).

\begin{table}[htbp]
\centering
\caption{The coefficients of \(B\) and \(U\).}
\small
\renewcommand{\arraystretch}{1.65}
\begin{tabular}{@{}cr@{}}
\toprule
Parameter & Exact value\\
\midrule
\(\ell\) & \(\frac{17234670337576735171}{20000000000000000000}\)\\
\(b_1\) & \(\frac{-91762684324025364521}{50000000000000000000}\)\\
\(b_2\) & \(\frac{2630271982013917217}{6250000000000000000}\)\\
\(b_3\) & \(\frac{-7544371695907842207}{50000000000000000000}\)\\
\(b_4\) & \(\frac{1412826035912896311}{25000000000000000000}\)\\
\(b_5\) & \(\frac{-1900116405707627717}{100000000000000000000}\)\\
\(b_6\) & \(\frac{106820616638860551}{25000000000000000000}\)\\
\(a_1\) & \(\frac{-683664790512221}{5000000000000000}\)\\
\(a_2\) & \(\frac{210251029030421}{2500000000000000}\)\\
\(a_3\) & \(\frac{-2136926142618817}{50000000000000000}\)\\
\(a_4\) & \(\frac{3815733882548219}{250000000000000000}\)\\
\bottomrule
\end{tabular}
\end{table}

For the preprocessing polynomial, put
\[
 D_P=54945054945054952265606000000000000,
 \qquad p_k=N_k/D_P.
\]
Table~\ref{tab:preprocessing-coefficients} lists the integer numerators.
Every even coefficient and every coefficient of degree greater than
\(81\) is zero. Exact summation gives
\(\sum_k|N_k|/D_P=999999999999/10^{12}\).

{\small
\begin{longtable}{@{}r r@{}}
\caption{Integer numerators of \(P\) in the common denominator \(D_P\).}
\label{tab:preprocessing-coefficients}\\
\toprule
\(k\) & \(N_k\)\\
\midrule
\endfirsthead
\toprule
\(k\) & \(N_k\)\\
\midrule
\endhead
\midrule
\multicolumn{2}{r}{\emph{Continued on the next page}}\\
\endfoot
\bottomrule
\endlastfoot
1 & \(49757548951830099593905263000000000\)\\
3 & \(-5181796689779939082431094000000000\)\\
5 & \(-3586456256823087170116799700\)\\
7 & \(432337220020275080362396875\)\\
9 & \(-77515986037541231266709634\)\\
11 & \(1813839253266558688219259100000\)\\
13 & \(-2111137078429535566218271650000\)\\
15 & \(109421788407171787002801510000\)\\
17 & \(258852989211570707450348610000\)\\
19 & \(-334257009058224671562370012500\)\\
21 & \(-86572484132477861581801501500\)\\
23 & \(186959578861277397563995678500\)\\
25 & \(42848401208951085664725072000\)\\
27 & \(-119641424015011875059501352000\)\\
29 & \(-54637263332040708890514058500\)\\
31 & \(69975395035262508121448184000\)\\
33 & \(72127361578106268242817426000\)\\
35 & \(-14519991251763254734022947500\)\\
37 & \(-63589301067645564545086137000\)\\
39 & \(-31015074144899721189591333000\)\\
41 & \(26944867422365961725984200500\)\\
43 & \(44887219130162882011639023000\)\\
45 & \(15303791586330027527094973350\)\\
47 & \(-22361224278899710204292506500\)\\
49 & \(-31678667932439222430968628000\)\\
51 & \(-10432152218171595320309290500\)\\
53 & \(16903692015101239165127739000\)\\
55 & \(26362279527849819039153442500\)\\
57 & \(13492892529412034579998047000\)\\
59 & \(-7886781300143445086831716800\)\\
61 & \(-20247479541207911861118112500\)\\
63 & \(-15907167543831625799489119500\)\\
65 & \(-373012574341125503909174985\)\\
67 & \(14098934856120450515693316000\)\\
69 & \(17969703421186799527348026000\)\\
71 & \(10073117824180530178879106100\)\\
73 & \(-3174665796161867092445593200\)\\
75 & \(-13041517002967832109370645500\)\\
77 & \(-14082134383563552481001037750\)\\
79 & \(-6747313828606468620237320850\)\\
81 & \(3844072792111654278856236150\)\\
\end{longtable}
}

\bibliographystyle{plain}
\bibliography{references}
\end{document}